\documentclass[12pt]{amsart}
\usepackage{amssymb}
\usepackage{amsmath}
\usepackage{amssymb}
\usepackage{amsthm}
\usepackage{enumerate}
\usepackage[all,arc,curve,color,frame]{xy}

\newcommand{\A}{{\mathcal A}}

\newcommand{\I}{{\mathcal I}}

\renewcommand{\L}{{\mathcal L}}
\newcommand{\M}{{\mathcal M}}

\newcommand{\C}{\ensuremath{\mathbb{C}}}

\newcommand{\p}{\partial}

\newcommand{\Z}{\ensuremath{\mathbb{Z}}}
\newtheorem{lemma}{Lemma}
\newtheorem{proposition}{Proposition}
\newtheorem{theorem}{Theorem}
\newtheorem{corollary}{Corollary}
\begin{document}

\title[Quantization on a super-K\"ahler manifold]
{Deformation quantization with separation of variables on a super-K\"ahler manifold}
\author[Alexander Karabegov]{Alexander Karabegov}
\address[Alexander Karabegov]{Department of Mathematics, Abilene
Christian University, ACU Box 28012, Abilene, TX 79699-8012}
\email{axk02d@acu.edu}

\begin{abstract}
We construct deformation quantizations with separation of variables on a split super-K\"ahler manifold and describe their canonical supertrace densities.
\end{abstract}
\date{March 14, 2016}
\maketitle

\section{Introduction}

Deformation quantization of Poisson manifolds via star products was introduced in \cite{BFFLS}. The existence of star products on arbitrary Poisson manifolds and their classification up to equivalence was given by Kontsevich in \cite{K}. Fedosov gave  in \cite{F1} a geometric construction of star products in each equivalence class on an arbitrary symplectic manifold. Deformation quantizations with separation of variables on pseudo-K\"ahler manifolds related to Berezin's quantization were introduced and classified in \cite{CMP1}, \cite{BW}, and \cite{N}. There are many papers on geometric, symbol, and deformation quantization on supermanifolds (see \cite{B}, \cite{KB}, \cite{V}, \cite{EN}, \cite{BKLR}). In this paper we describe a class of star products with separation of variables on a split complex supermanifold $\Pi E$, where $E$ is a holomorphic vector bundle on a pseudo-K\"ahler manifold $M$. Below we give a number of definitions which will be used throughout this paper.

Let $\A$ be a supercommutative unital $\Z_2$-graded associative algebra with the identity 1. Given an element $f \in \A$, we denote the operator of left multiplication by $f$ by the same symbol. An operator $A$ on $\A$ is a (left) differential operator if there exists a nonnegative integer $n$ such that for any elements $f_0, \ldots f_n \in A$, one has
\[
     [f_n,  [f_{n-1}, \ldots  [f_0, A] \ldots]] = 0,
\]
where $\Z_2$-graded commutators are used. The smallest $n$ with this property is called the order of $A$. This algebraic definition gives the ordinary differential operators on the algebra of smooth functions on a manifold.

Denote by $\A[\nu^{-1}, \nu]]$ the space of formal Laurent series with a finite principal part with coefficients from $\A$,
\[
               f = \nu^r f_r + \nu^{r+1} f_{r+1} + \ldots,
\]
where $r \in \Z$ and $f_k \in \A$ for $k \geq r$. A formal differential operator on $\A[\nu^{-1}, \nu]]$ is a formal series
\begin{equation}\label{E:fdo}
                  A = \nu^r A_r + \nu^{r+1} A_{r+1} + \ldots,
\end{equation}
where $r \in \Z$ and $A_k$ is a differential operator on $\A$ for $k \geq r$. A formal differential operator (\ref{E:fdo}) is called natural if $r = 0$ and the order of $A_k$ is not greater than $k$. For a natural $A$, $A_0$ is a left multiplication operator by an element of $\A$.

Let $\ast$ be a $\nu$-linear $\nu$-adically continuous $\Z_2$-graded associative product on $\A[\nu^{-1}, \nu]]$ with the identity 1, given by the formula
\begin{equation}\label{E:star}
              f \ast g = \sum_{r=s}^\infty \nu^r C_r(f,g),
\end{equation}
where $s$ is a fixed nonpositive integer and the mapping $C_r : \A \times \A \to \A$ is extended to $\A[\nu^{-1}, \nu]]$ by $\nu$-linearity. 

Let $L_f$ be the left $\ast$-multiplication operator by $f$ on $\A[\nu^{-1}, \nu]]$, so that $L_f g = f \ast g$. We denote by $R_f$ the graded right $\ast$-multiplication operator by $f$ defined for homogeneous $f, g$ as follows,
\[
                              R_f g = (-1)^{|f||g|} g \ast f.
\]
Then
\[
                L_f g - R_f g = [f,g]_\ast,
\]
where $[\cdot,\cdot]_\ast$ is the supercommutator on the algebra $(\A[\nu^{-1}, \nu]],\ast)$. The operators $L_f$ and $R_g$ supercommute for any $f,g \in \A[\nu^{-1}, \nu]]$. 

A product $\ast$ is called differential if the operators $L_f, R_f$ are (left) formal differential operators for any $f \in \A$. A differential product  (\ref{E:star}) on $\A[\nu^{-1}, \nu]]$ is a star product if it is a formal deformation of the supercommutative product on $\A$, i.e., if $s=0$ and $C_0(f,g) = fg$. For a star product $\ast$, the bidifferential operator
\[
     f,g \mapsto C_1(f,g) - (-1)^{|f||g|} C_1(g,f)
\]
is a Poisson bracket on $\A$. It is a right derivation in $f$ and a left derivation in $g$, it is graded antisymmetric, and satisfies the graded Jacobi identity.

Two star products $\ast$ and $\ast'$ on $\A[\nu^{-1}, \nu]]$ are equivalent if there exists a formal differential operator $T = 1 + \nu T_1 + \nu^2 T_2 + \ldots$ on $\A[\nu^{-1}, \nu]]$ such that
\[
      f \ast' g = T^{-1} (Tf \ast Tg).
\]
 A star product $\ast$ is called natural if the operators $L_f$ and $R_f$ are natural for any $f \in \A[[\nu]]$. Equivalently, $\ast$ is natural  if the operators $L_f$ and $R_f$ are natural for any $f \in \A$.
The concept of a natural star product was introduced in  \cite{GR}.

{\bf Acknowledgments.} I want to express my gratitude to H. Khudaverdian and Th. Voronov for answering my numerous questions on supermathematics.

\section{Deformation quantization on Poisson manifolds}

A star product $\star$ on a manifold $M$ is a formal differential deformation of the commutative product on $C^\infty(M)$.  It induces a Poisson bracket on $M$,
\[
     \{f,g\} := -i (C_1(f,g) - C_1(g,f)).
\]
A deformation quantization of a Poisson manifold $(M, \{\cdot, \cdot\})$ is given by a star product on $M$ which induces the  Poisson bracket $\{\cdot,\cdot\}$. Kontsevich proved in \cite{K} that the equivalence classes of star products on a Poisson manifold $M$ are parametrized by the formal deformations of the Poisson structure. 

A star product on a manifold $M$ can be restricted to any open subset $U \subset M$. Hence, one can consider star products of locally defined functions. 

We say that a formal function $f = \nu^r f_r + \nu^{r+1}f_{r+1} +\ldots$ on $M$ has compact support if each $f_k, k \geq r$, has compact support (but we do not require that the supports of all $f_k$ are contained in a compact set).

We will call a star product on a manifold $M$ nondegenerate if it induces a nondegenerate Poisson structure on $M$. The nondegenerate Poisson tensors (bivector fields) on a manifold $M$ bijectively correspond to the symplectic forms on $M$. The equivalence classes of star products on a symplectic manifold $M$ equipped with a symplectic form $\omega_{-1}$ are bijectively parametrized by the formal de Rham cohomology classes from
\[
      \nu^{-1}[\omega_{-1}] + H^2(M)[[\nu]].
\]
For each star product $\star$ on a symplectic manifold $(M,\omega_{-1})$ of dimension $2m$ there exists a canonically normalized formal trace density
\begin{equation}\label{E:trdens}
                   \mu_\star = \frac{1}{m! \, \nu^m} \left(\omega_{-1}\right)^m e^\varkappa, 
\end{equation}
where $\varkappa = \nu \varkappa_1 + \nu^2 \varkappa_2 + \ldots$ is a formal function globally defined on ~$M$. The canonical normalization of (\ref{E:trdens}) can be described intrinsically as follows (see \cite{LMP2}). Let $U$ be a contractible open subset of $M$. There exists a formal derivation
\[
               \delta = \frac{d}{d\nu} + A
\]
of the star product $\star$ on $U$, where $A = \nu^{-1} A_{-1} + A_0 + \nu A_1 + \ldots$ is a formal differential operator on $U$. The canonical trace density $\mu_\star$ is uniquely determined on $U$ by the condition that
\[
       \frac{d}{d\nu} \int_U f \, \mu_\star = \int_U \delta(f) \, \mu_\star
\]
for any $f \in C^\infty(U)[\nu^{-1},\nu]]$ with compact support. If $M$ is compact, the index theorem for deformation quantization (\cite{F2}, \cite{NT}) expresses the total volume
\[
              \int_M \mu_\star
\]
via a topological formula involving the cohomology class of the star product. 

Let $M$ be a complex manifold. A star product $\star$ on $M$ satisfying the conditions
\begin{equation}\label{E:separ}
               a \star f = af \mbox{ and } f \star b = fb
\end{equation}
for any locally defined holomorphic function $a$ and antiholomorphic function $b$ is called a star product with separation of variables (of the anti-Wick type). Conditions (\ref{E:separ}) mean that $L_a = a$ and $R_b = b$ are point-wise multiplication operators. Equivalently, the bidifferential operators $C_r$ differentiate their first argument in antiholomorphic directions and their second argument in holomorphic ones. The Poisson tensor corresponding to a star product with separation of variables $\star$ is of type (1,1) with respect to the complex structure. In local coordinates the operator $C_1$ and the Poisson bracket for the product $\star$ are of the form
\[
     C_1(f,g) = g^{lk} \frac{\p f}{\p \bar z^l} \frac{\p g}{\p z^k} \mbox{ and } \{f,g\} = i g^{lk} \left(\frac{\p f}{\p z^k}\frac{\p g}{\p \bar z^l} - \frac{\p g}{\p z^k}\frac{\p f}{\p \bar z^l}\right),
\]
where $g^{lk}$ is the corresponding Poisson tensor. Here, as well as in the rest of the paper, we assume summation over repeated lower and upper indices. If the Poisson tensor $g^{lk}$ is nondegenerate, its inverse $g_{kl}$ is a pseudo-K\"ahler metric tensor.

It was proved in \cite{CMP1} that the star products with separation of variables on a pseudo-K\"ahler manifold $M$ with a pseudo-K\"ahler form $\omega_{-1}$ are bijectively parametrized by the closed formal (1,1)-forms
\begin{equation}\label{E:omega}
      \omega = \nu^{-1}\omega_{-1} + \omega_0 + \nu \omega_1 + \ldots
\end{equation}
on $M$. The form $\omega$ parametrizing a star product with separation of variables $\star$ is called its classifying form. The existence of star products with separation of variables (of the Wick type) on arbitrary pseudo-K\"ahler manifolds was shown in \cite{BW} using a generalization of Fedosov's construction. It was proved in \cite{N} that every star product with separation of variables on a pseudo-K\"ahler manifold can be obtained via Fedosov's approach.

Given a formal form (\ref{E:omega}) on a pseudo-K\"ahler manifold $(M, \omega_{-1})$, the star product with separation of variables $\star$ with the classifying form $\omega$ on~ $M$ is completely characterized by the following property. Let $U$ be any contractible coordinate chart on $M$ and $\Phi_r$ be a potential of the form~ $\omega_r$ on $U$ for $r \geq -1$, i.e., $\omega_r = i \p \bar \p \Phi_r$. Then
\begin{equation}\label{E:classleft}
           L_{\frac{\p\Phi}{\p z^k}} = \frac{\p\Phi}{\p z^k} + \frac{\p}{\p z^k},
\end{equation}
where
\[
     \Phi : = \nu^{-1}\Phi_{-1} + \Phi_0 + \nu \Phi_1 + \ldots
\]
is a formal potential of $\omega$ on $U$. The star product $\star$ is also characterized by the property that
\begin{equation}\label{E:classright}
      R_{\frac{\p\Phi}{\p \bar z^l}} = \frac{\p\Phi}{\p \bar z^l} + \frac{\p}{\p \bar z^l}.
\end{equation}

Given a (possibly degenerate) star product with separation of variables $\star$ on a complex manifold $M$, we define a formal differential operator
\[
\I_\star = 1 + \nu \I_1 + \nu^2 \I_2 +\ldots
\]
on $M$ as follows:
\[
                   \I_\star (ba) = b \star a,
\]
where $a$ and $b$ are local holomorphic and antiholomorphic functions, respectively. The operator $\I_\star$ is globally defined on $M$. It is called the formal Berezin transform of the star product $\star$. A star product with separation of variables is completely determined by its formal Berezin transform. It was proved in \cite{CMP3} that the star product
\[
                   f \star' g := \I_\star^{-1} (\I_\star f \star \I_\star g) 
\]
is a star product with separation of variables of the Wick type (i.e., with the r\^oles of holomorphic and antiholomorphic coordinates swapped) on the same Poisson manifold. Its opposite product
\[
     f \tilde \star g : = \I_\star^{-1} (\I_\star g \star \I_\star f) 
\]
is a star product with separation of variables of the anti-Wick type on $M$ equipped with the opposite Poisson structure. We call $\tilde\star$ the star product dual to $\star$. Its formal Berezin transform is $\I_{\tilde\star} = \I_\star^{-1}$. The star product $\star$ is dual to $\tilde\star$.

If $\star$ is a star product with separation of variables on a pseudo-K\"ahler manifold $(M,\omega_{-1})$ with a classifying form $\omega$, one can construct a trace density of the product $\star$ on a contractible coordinate chart $U \subset M$ as follows (see \cite{LMP2} and \cite{JGP}).  Let $\Phi$ be a formal potential of $\omega$ on $U$. There exists a formal potential $\Psi$ on $U$ of the classifying form $\tilde\omega$ of the dual star product $\tilde\star$ satisfying the equations
\[
          \frac{\p\Phi}{\p z^k} + \I_\star \left(\frac{\p\Psi}{\p z^k}\right) = 0 \mbox{ and } \frac{\p\Phi}{\p \bar z^l} + \I_\star \left(\frac{\p\Psi}{\p \bar z^l}\right) = 0,
\]
which determine $\Psi$ up to an additive formal constant. Then
\[
                    e^{\Phi+\Psi} dz d\bar z,
\]
where $dzd\bar z$ is a Lebesgue measure on $U$, is a trace density for the star product $\star$ on $U$. One can canonically normalize it as follows. The form $\omega_{-1}$ can be written on $U$ as
\[
                \omega_{-1} =i \p \bar \p \Phi_{-1} = i g_{kl} dz^k \wedge d \bar z^l,
\]
where the metric tensor 
\[
g_{kl} = \frac{\p^2 \Phi_{-1}}{\p z^k \p \bar z^l}
\] 
is inverse to the Poisson tensor $g^{lk}$. Set $\mathbf{g} = \det \left(g_{kl}\right)$ and choose a branch $\log \mathbf{g}$ of the logarithm of $\mathbf{g}$ on $U$.  There exists a unique potential $\Psi$ of the form $\tilde\omega$ on $U$ such that
\[
   \Psi = - \nu^{-1}\Phi_{-1} + (- \Phi_0 + \log \mathbf{g}) + \nu \Psi_1 + \ldots,
\]
which satisfies the equation
\[
      \frac{d\Phi}{d\nu} + \I_\star \left(\frac{d\Psi}{d\nu}\right) = \frac{m}{\nu}.
\]
The function $\varkappa$ from (\ref{E:trdens}) is given on $U$ by the formula
\[
     \varkappa = \Phi+\Psi - \log\mathbf{g} = \nu(\Phi_1 + \Psi_1) + \ldots.
\]
In this paper we generalize some of the abovementioned constructions including that of a canonical trace density to the case of a star product with separation of variables on $\Pi E$. 

\section{A product on a split supermanifold}\label{S:split}

Let $E$ be a holomorphic vector bundle of rank $d$ on a complex manifold $M$ of complex dimension $m$. In this section we construct a differential product on the formal functions on the split supermanifold $\Pi E$, which is not necessarily a deformation of the supercommutative product on $C^\infty(\Pi E)$. 

Let $U \subset M$ be a coordinate chart with coordinates $\{z^k, \bar z^l\}$ such that $E$ is holomorphically trivial over $U$,
$E|_U \cong U \times \C^d$ be a holomorphic trivialization of $E$, and $\{\theta^\alpha, \bar \theta^\beta\}$ be the odd fiber coordinates on the corresponding trivialization $\Pi E|_U \cong U \times \C^{0|d}$.  Denote by $[d]$ the ordered set of integers $\{1,2, \ldots,d\}$. We consider the ordered subsets $I = \{\alpha_1, \ldots ,\alpha_k\} \subset [d]$, where $1 \leq \alpha_1 < \ldots < \alpha_k \leq d$, as tensor indices, and set $|I| = k$ and
\[
                \theta^I := \theta^{\alpha_1} \ldots \theta^{\alpha_k}.
\]
A function on $\Pi E$ is a section of the bundle $\wedge E^\ast$ of Grassmann algebras. We denote by $C^\infty(\Pi E): = C^\infty(\wedge E^\ast)$ the space of functions on $\Pi E$. The restriction of a function $f$ on $\Pi E$ to $\Pi E|_U$ is identified with a sum
\[
                  f = f_{IJ} \theta^I \bar \theta^J,
\]
where $f_{IJ} \in C^\infty(U)$ and summation over repeated tensor indices is assumed. Suppose there is another holomorphic trivialization of $E|_U$ with the holomorphic transition functions $a^\alpha_\gamma$ on $U$. We set $b^\alpha_\gamma := \overline{a^\alpha_\gamma}$. Let $\eta^\alpha, \bar \eta^\beta$ be the odd fiber coordinates on $\Pi E|_U$ corresponding to the second trivialization. Then
\[
                                 \theta^\alpha = a^\alpha_\gamma \eta^\gamma \mbox{ and } \bar\theta^\beta = b^\beta_\delta \bar \eta^\delta.
\]
The transition functions $a^\alpha_\gamma$ and $b^\beta_\delta$ induce holomorphic matrices $a^I_K$ and antiholomorphic matrices $b^J_L$ such that
\[
                      \theta^I = a^I_K \eta^K \mbox{ and } \bar\theta^J = b^J_L \bar\eta^L.
\]
We have $f = f_{IJ} \theta^I \bar \theta^J = f'_{KL} \eta^K \bar\eta^L,$ where $f'_{KL} = f_{IJ}a^I_K b^J_L$.

We call functions $a = a_I \theta^I \mbox{ and } b = b_J \bar \theta^J$ on $U \times \C^{0|d}$ holomorphic and antiholomorphic if $a_I$ and $b_J$ are holomorphic and antiholomorphic functions on $U$, respectively. 

We will use arrows to indicate that  multiplication operators by Grassmann variables and partial derivatives with respect to these variables act from the left or from the right. By default, we assume that they act from the left and omit the arrows.
Given $I= \{\alpha_1,\ldots ,\alpha_k\}$, we set
\[
   {\frac{\p}{\p \theta^I}} := {\frac{\p}{\p \theta^{\alpha_k}}} \ldots {\frac{\p}{\p \theta^{\alpha_1}}}.
\]
We say that an operator $A$ on $C^\infty(U \times \C^{0|d})$ is Grassmann if it is given by a matrix $A_{KL}^{IJ}$ with constant entries and acts on $f = f_{IJ}\theta^I \bar \theta^J$ as follows,
\[
                      Af = f_{IJ}A_{KL}^{IJ}\, \theta^K \bar \theta^L.
\]
Holomorphic and antiholomorphic functions are well defined on $\Pi E|_U$, but Grassmann operators depend on the trivialization. We call an operator $A$ holomorphic if $A_{KL}^{IJ} = A_K^I \delta_L^J$ and graded antiholomorphic if
\[
                  A_{KL}^{IJ} = (-1)^{|I|(|J|+|L|)}A_L^J \delta_K^I.
\]
Denote by $\delta_K$ and $\bar \delta_L$ the Grassman operators acting on $f = f_{IJ}\theta^I \bar \theta^J$ as follows,
\[
   \delta_K: f_{IJ}\theta^I \bar \theta^J \mapsto f_{KJ}\bar \theta^J, \bar \delta_L:f_{IJ}\theta^I \bar \theta^J \mapsto (-1)^{|I||L|} f_{IL} \theta^I.
\]
We have
\begin{equation}\label{E:delk}
     \delta_K f = {\frac{\p}{\p \theta^K}} f\Big |_{\theta=0} \mbox{ and } \bar \delta_L f = \frac{\p}{\p \bar \theta^L} f \Big |_{\bar \theta=0}.
\end{equation}
We will use the bases
\[
      \left\{\theta^I \delta_K \right\} \mbox{ and } \left\{ \theta^I \frac{\p}{\p \theta^K}\right\}
\]
in the algebra of holomorphic Grassmann operators and the bases
\[
      \left\{\bar\theta^J \bar\delta_L \right\} \mbox{ and } \left\{ \bar\theta^J \frac{\p}{\p \bar \theta^L}\right\}
\]
in the algebra of graded antiholomorphic Grassmann operators. 

In order to define a product on the formal functions on $U \times \C^{0|d}$, we fix a possibly degenerate star product with separation of variables $\star$ on $U$ and an even element 
\begin{equation}\label{E:upq}
u = u_{PQ} \theta^P \bar \theta^Q \in C^\infty(U \times \C^{0|d})[\nu^{-1}\nu]]
\end{equation}
with $u_{\emptyset\emptyset} = 1$, so that $u-1$ is nilpotent. We call the function $u$ admissible with respect to the star product $\star$  if the matrix $(u_{PQ})$ has an inverse over the algebra $(C^\infty(U)[\nu^{-1},\nu]], \star)$, i.e., there exists a matrix $(v^{QP})$ with the entries from $C^\infty(U)[\nu^{-1},\nu]]$ such that
\begin{equation}\label{E:uinv}
u_{PQ} \star v^{QK} = \delta_P^K \mbox{ and } v^{LP} \star u_{PQ} = \delta^L_Q.
\end{equation}
Given $f \in C^\infty(U)[\nu^{-1},\nu]]$, we denote by $L_f^\star$ and $R_f^\star$ the operators of left and right $\star$-multiplication by $f$, respectively. We extend these operators to the space $C^\infty(U\times\C^{0|d})[\nu^{-1},\nu]]$ assuming that they commute with multiplication by the Grassmann variables. Let $\M_\star(U)$ be the $2^d \times 2^d$-matrix algebra over the algebra $(C^\infty(U)[\nu^{-1},\nu]], \star)$ with the matrix entries indexed by the tensor indices $I \subset [d]$. We consider the following homomorphism from $\M_\star(U)$ to the algebra of operators on $C^\infty(U \times \C^{0|d})[\nu^{-1},\nu]]$,
\begin{equation}\label{E:matralg}
        \M_\star(U) \ni (f_K^I) \mapsto u^{-1} \left(L^\star_{f_K^I} \theta^K \delta_I \right) u,
\end{equation}
where, by abuse of notations, we denoted by $u$ the multiplication operator by the element $u$ and by $u^{-1}$ its inverse. Applying the (target) operator from (\ref{E:matralg}) to the unit constant, we get a mapping from $\M_\star(U)$ to $C^\infty(U \times \C^{0|d})[\nu^{-1},\nu]]$,
\begin{equation}\label{E:mop}
     (f_K^I) \mapsto \left\{u^{-1} \left(L^\star_{f_K^I} \theta^K \delta_I \right) u\right\} 1 = u^{-1} (f_K^I \star u_{IL})\theta^K \bar \theta^L.
\end{equation}
The mapping (\ref{E:mop}) is a bijection if and only if the function $u$ is admissible with respect to the star product $\star$.

Now assume that the function $u$ is admissible. Denote by $\ast$ the product on $C^\infty(U \times \C^{0|d})[\nu^{-1},\nu]]$ transferred from the algebra $\M_\star(U)$ via (\ref{E:mop}) and let $L_f$ and $R_f$ be the operators of left and graded right $\ast$-multiplication by a formal function $f$ on $U \times \C^{0|d}$, respectively. Then,
\begin{equation}\label{E:oplf}
                     L_f = u^{-1} \left(L^\star_{f_K^I} \theta^K \delta_I \right) u \mbox{ for } f = u^{-1} (f_K^I \star u_{IL})\theta^K \bar \theta^L.
\end{equation}
It follows that
\begin{equation}\label{E:prod}
                 f \ast g = u^{-1}( (uf)_{KQ} \star v^{QP} \star (ug)_{PL})\theta^K \bar \theta^L,
\end{equation}
where $(uf)_{KL}$ is the formal function on $U$ such that $uf = (uf)_{KL}\theta^K \bar \theta^L$. Also, we have that
\begin{eqnarray}\label{E:oprt}
      R_f = u^{-1} \left(R^\star_{f_L^J} \bar \theta^L \bar \delta_J \right) u\hskip 4cm \\
\mbox{ for }
f = (-1)^{|K|(|J|+|L|)}u^{-1} (u_{KJ} \star f_L^J)\theta^K \bar \theta^L. \nonumber
\end{eqnarray}

Given a star product with separation of variables $\star$ on an open set $U \subset \C^m$ and an admissible function $u$ on 
$U \times \C^{0|d}$, we say that the corresponding product $\ast$ on $U \times \C^{0|d}$ is associated with the pair $(\star,u)$.

A differential product $\ast$ on a split supermanifold $\Pi E$ has the property of separation of variables if for any locally defined holomorphic function $a = a_I \theta^I$ and antiholomorphic function $b = b_J\bar\theta^J$,
\[
                     a \ast f = af \mbox{ and } f \ast b = fb.
\]
It means that both $L_a$ and $R_b$ are {\it left} multiplication operators, $L_a = a$ and $R_b = b$. Clearly, $L_a f = a \ast f = af$. For homogeneous $f$ and $b$ we have
\[
        R_b f = (-1)^{|f| |b|} f \ast b = (-1)^{|f| |b|} fb = bf.
\]

\begin{lemma}\label{L:separ}
Given a possibly degenerate star product with separation of variables $\star$ on an open set $U \subset \C^m$ and an admissible function $u$ on 
$U \times \C^{0|d}$, the product $\ast$ on $U \times \C^{0|d}$ associated with the pair $(\star,u)$ has the property of separation of variables.
\end{lemma}
\begin{proof}
It follows from (\ref{E:oplf}) that the operator 
\[
u^{-1}\left(L^\star_{a_I} {\theta^I}\right) u = a_I {\theta^I} = {a}
\]
is the operator of left $\ast$-multiplication by $a$  which coincides with the operator of left multiplication by $a, \ L_a = {a}$. Similarly, one can derive from (\ref{E:oprt}) that  the graded right $\ast$-multiplication operator $R_b$ is the left multiplication operator by $b$, $R_b =b$.
\end{proof}

\begin{lemma}\label{L:eab} Let $\star$ be a possibly degenerate star product with separation of variables on an open set $U \subset \C^m$ and $u$ be an admissible function with respect to $\star$ on $U \times \C^{0|d}$. If $a$ and $b$ are even nilpotent  formal functions on $U \times \C^{0|d}$ such that $a$ is holomorphic and $b$ is antiholomorphic, then the function
\begin{equation}\label{E:eab}
              \tilde u = e^{a+b}u
\end{equation}
is also admissible  with respect to $\star$. Moreover, the products on $U \times \C^{0|d}$ associated with the pairs $(\star,u)$ and  $(\star, \tilde u)$ coincide.
\end{lemma}
\begin{proof}
Given an even nilpotent formal antiholomorphic function $b$ on $U \times \C^{0|d}$, the mapping (\ref{E:matralg})  will not change if we replace the function $u$ with $e^b u$. Hence, the function $e^b u$ is admissible and the products on $U \times \C^{0|d}$ associated with the pairs $(\star, u)$ and $(\star, e^bu)$ coincide. Let $a$ be an even nilpotent formal holomorphic function on $U \times \C^{0|d}$. Then $L_{e^a} = e^a$ is the operator of multiplication by $e^a$. Therefore, the algebra of left $\ast$-multiplication operators is invariant under conjugation by the multiplication operator by $e^a$. If the function $u$ is replaced with $e^au$, then the mapping (\ref{E:mop}) is also bijective. It follows that the function $e^{a}u$ is admissible and the products associated with the pairs $(\star, u)$ and $(\star, e^au)$ coincide.
\end{proof}
Since $\delta_I$ and $\bar \delta_J$ are (left) differential operators, it follows from (\ref{E:oplf}) and (\ref{E:oprt}) that the  product (\ref{E:prod}) is bidifferential. Since $u_{KL}, v^{LK} \in C^\infty(U)[\nu^{-1},\nu]]$, it can be written as (\ref{E:star}). Then (\ref{E:star}) is a star product if $s = 0$ and $C_0(f,g) = fg$.

{\it Example.} Let $U = \{pt\}$ be a point (so that $m=0$), $d =1$, and $u = 1 + \nu^{-n}\theta\bar\theta$. The corresponding product  $\ast$ on $\C^{0|1}$ satisfies $\bar \theta \ast \theta = \bar \theta \theta + \nu^n$. It is a star product if $n \geq 1$. If $n=0$, then $C_0(\bar\theta, \theta) =  \bar\theta \theta + 1$.

Given a complex manifold $M$, let $\star$ be a star product with separation of variables on a coordinate chart $U \subset M$ and $E$ be a holomorphic vector bundle on $M$ holomorphically trivializable over $U$. One can define an admissible function $u$ on $\Pi E|_U$ and a product $\ast$ on $\Pi E|_U$ associated with $(\star, u)$ using the identification of $\Pi E|_U$ with $U \times \C^{0|d}$ via some holomorphic trivialization of $E|_U$. We will show that the product $\ast$ on $\Pi E|_U$ does not depend on the trivialization.

\begin{proposition}\label{P:trivindep}
Given a possibly degenerate star product with separation of variables $\star$ on $U \subset M$ and a function $u \in C^\infty(\Pi E|_U)[\nu^{-1}, \nu]]$ which is admissible for some holomorphic trivialization of $\Pi E|_U$, then $u$ is admissible for any holomorphic trivialization of $\Pi E|_U$ and the corresponding product $\ast$ does not depend on the trivialization.
\end{proposition}
\begin{proof}
Consider two holomorphic trivializations of $\Pi E|_U$ with odd fiber coordinates $\theta^k, \bar\theta^l$ and $\eta^k, \bar\eta^l$, respectively, and assume that $u$ is admissible for the former one. The holomorphic transition functions of $E$ over $U$ induce invertible matrices with holomorphic entries $a_K^I$ and matrices with antiholomorphic entries $b_L^J$ from $C^\infty(U)$ such that
\begin{equation}\label{E:etatheta}
           \theta^I = a^I_K \eta^K \mbox{ and } \bar\theta^J = b_L^J \bar\eta^L.
\end{equation}
We will denote their inverse matrices by $\tilde a_K^I$ and $\tilde b_L^J$, respectively. These matrices are also their inverses over the algebra $(C^\infty(U)[\nu^{-1},\nu]],\star)$,
\begin{equation}\label{E:ainv}
     a_K^I \tilde a_P^K =  a_K^I \star \tilde a_P^K = \delta^I_P,  b_L^J \tilde b_Q^L =  b_L^J \star \tilde b_Q^L = \delta^J_Q.
\end{equation}
A function $f$ on $\Pi E|_U$ can be written in coordinates as follows, $f=f_{IJ}\theta^I \bar\theta^J =  f'_{KL}\eta^K \bar\eta^L$, where $f_{IJ}$ and $f'_{KL}$ are functions on $U$.  Then
\begin{equation}\label{E:ffpr}
                  f'_{KL} = f_{IJ} a^I_K b^J_L =  a^I_K \star f_{IJ} \star b^J_L.
\end{equation}
It follows from (\ref{E:ffpr}) for $f = u$ that the inverse matrix of $u'_{KL}$ over the algebra $(C^\infty(U)[\nu^{-1},\nu]],\star)$ is
\begin{equation}\label{E:wv}
            w^{LK} := \tilde b^L_Q \star v^{QP} \star \tilde a^K_P.
\end{equation}
Thus $u$ is admissible for the second trivialization. Consider the product 
\[
         f \circ g := (f_{IQ} \star v^{QP} \star g_{PJ}) \theta^I \bar \theta^J
\]
equivalent to the product (\ref{E:prod}) via the mapping $f \mapsto uf$. In order to prove the proposition it suffices to prove that the product $\circ$
does not depend on the trivialization. Set
\[
               h_{IJ}: = f_{IQ} \star v^{QP} \star g_{PJ},  \mbox{ so that }  f \circ g = h_{IJ}\theta^I \bar\theta^J.
\]
We have from (\ref{E:ainv}), (\ref{E:ffpr}), and (\ref{E:wv}) that
\begin{eqnarray*}
   f'_{KQ} \star w^{QP} \star g'_{PL} = a^I_K \star f_{IJ} \star b^J_S \star \tilde b^S_Q \star v^{QP} \star \tilde a^R_P \star  a^I_R \star g_{IJ} \star b^J_L =\\
a^I_K \star f_{IQ} \star v^{QP} \star g_{PJ} \star b^J_L = h'_{KL},
\end{eqnarray*}
whence the proposition follows.
\end{proof}
The following corollary is immediate.
\begin{corollary}
    Given a possibly degenerate star product with separation of variables $\star$ on $M$ and a globally defined admissible function $u \in C^\infty(\Pi E)[\nu^{-1}, \nu]]$, then the corresponding product $\ast$ is globally defined on the functions on $\Pi E$.
\end{corollary}

Let $\star$ be a star product with separation of variables on a pseudo-K\"ahler manifold $M$, $E$ be a holomorphic vector bundle on $M$ holomorphically trivializable over $U \subset M$, $u \in C^\infty(\Pi E|_U)[\nu^{-1}, \nu]]$ be admissible with respect to $\star$, and $\ast$ be the product on $C^\infty(\Pi E|_U)[\nu^{-1}, \nu]]$ associated with $(\star, u)$. In the rest of this section we define a canonical supertrace functional $\sigma$ on the algebra $( C^\infty(\Pi E|_U)[\nu^{-1}, \nu]], \ast)$. 

Fix a holomorphic trivialization $\Pi E|_U \cong U \times \C^{0|d}$. We say that a formal function 
\[
f = \nu^r f_{r, IJ} \theta^I \bar\theta^J + \nu^{r + 1} f_{r + 1, IJ} \theta^I \bar\theta^J + \ldots
\]
on $\Pi E|_U$ has compact support if each coefficient $f_{r, IJ} \in C^\infty(U)$ has compact support (but there may be no common comact support for all $f_{r, IJ}$). Clearly, this notion does not depend on the trivialization.

We define a $\Z_2$-grading on the algebra $\M_\star(U)$ by setting the parity of a tensor index $I$ to be equal to the parity of $|I|$. The corresponding $\Z_2$-grading on the algebra $(C^\infty(U \times \C^{0|d})[\nu^{-1},\nu]],\ast)$ is the standard one given by the parity of its elements.

The star pro\-duct $\star$ has a canonical trace density $\mu_\star$ and there exists a supertrace on the elements of $\M_\star(U)$ with compact support defined as follows,
\[
       \mathrm{Str} (f_K^I) = \sum_I \int_U (-1)^{|I|}f_I^I \, \mu_\star.
\]
We define a supertrace $\sigma$ on the algebra $(C^\infty(\Pi E|_U)[\nu^{-1},\nu]],\ast)$ using a trivialization of $E|_U$. Given a formal function $f$ on $\Pi E|_U \cong U \times \C^{0|d}$ with compact support written as $f = u^{-1} (f_I^S \star u_{ST})\theta^I \bar \theta^T$ for some matrix $(f_K^I) \in \M_\star(U)$, we set
\[
      \sigma(f) = \mathrm{Str} (f_K^I).
\]

\begin{lemma}\label{L:sigind}
The functional $\sigma$ does not depend on the trivialization.
\end{lemma}
\begin{proof}
Let the odd coordinates $\theta^k, \bar\theta^l$ and $\eta^k, \bar\eta^l$ and the matrices $(a_K^I)$ and $(b_L^J)$ be as in the proof of Proposition \ref{P:trivindep} so that (\ref{E:etatheta}) holds. If $f$ is a formal function with compact support on $\Pi E|_U$, then
\[
       f = u^{-1} (f_I^S \star u_{ST})\theta^I \bar \theta^T = u^{-1} ({f'}_K^P \star {u'}_{PL})\eta^K \bar \eta^L
\]
for some matrices $(f_K^I)$ and   $({f'}_K^I)$, where $u_{IJ}$ and ${u'}_{KL}$ are connected  according to (\ref{E:ffpr}). Now,
\[
       (f_I^S \star u_{ST}) a_K^I b^T_L  =  {f'}_K^P \star (u_{ST} a^S_P b^T_L ).
\]
We have thus,
\[
     a_K^I \star f_I^S \star u_{ST} \star b^T_L = {f'}_K^P \star  a^S_P \star u_{ST} \star b^T_L,
\] 
whence
\[
           a_K^I \star f_I^S =    {f'}_K^P \star  a^S_P.
\]
The lemma follows from the fact that then $\mathrm{Str} (f_K^I) = \mathrm{Str} ({f'}_K^I)$.
\end{proof}

We want to show that the supertrace $\sigma$ for the product $\ast$ can be given by a Berezin supertrace density. Berezin integral on $\C^{0|d}$  is defined as follows,
\[
     \int f_{IJ} \theta^I \bar \theta^J\,  d\theta d\bar\theta = f_{[d] [d]}.
\]
Given $I= \{\alpha_1, \ldots, \alpha_k\} \subset [d]$, set $I' := [d] \setminus I$. Then we have
\[
     \theta^I \theta^{I'} = (-1)^{\varepsilon(I)} \theta^{[d]}, \mbox{ where } \varepsilon(I) := \alpha_1 + \ldots \alpha_k - \frac{k(k+1)}{2}.
\]
\begin{theorem}\label{T:trdens}
Let $\star$ be a nondegenerate star product with separation of variables  on an open subset $U \subset \C^m$,  $u$ be an admissible function on $U \times \C^{0|d}$, and $\ast$ be the product associated with the pair $(\star,u)$. There exists a unique canonical formal Berezin supertrace density 
\begin{equation}\label{E:murho}
     \mu  = \rho\,  dz d\bar z d\theta d\bar\theta
\end{equation}
for the product $\ast$, where $\rho \in C^\infty(U \times \C^{0|d})[\nu^{-1},\nu]]$ and $dz d\bar z$ is a Lebesgue measure on $U$, such that the canonical supertrace functional $\sigma$ of the product $\ast$ is given by the Berezin integral
\begin{equation}\label{E:ber}
               \sigma(f) = \int f \, \mu.
\end{equation}
In (\ref{E:ber}) $f$ is a formal function with compact support on $U \times \C^{0|d}$.
\end{theorem}
\begin{proof}
We will be looking for a function $\tau = \tau_{IJ}\theta^I \bar \theta^J$ such that
\[
       u^{-1} \rho\, dz d\bar z = \tau \mu_\star.
\]
For $f = u^{-1}(f_K^P \star u_{PQ})\theta^K \bar \theta^Q$ we have
\begin{eqnarray*}
    \sigma(f) =  \int (f_K^P \star u_{PQ})\theta^K \bar \theta^Q \, (\tau_{IJ}\theta^I \bar \theta^J) \mu_\star d\theta d\bar\theta =\\
    \sum_{K,Q} \int (f_K^P \star u_{PQ})\theta^K \bar \theta^Q \, (\tau_{K' Q'}\theta^{K'} \bar \theta^{Q'}) \mu_\star d\theta d\bar\theta =\\
\sum_{K,Q}\int (-1)^{\lambda(K,Q)} (f_K^P \star u_{PQ}) \tau_{K' Q'}\mu_\star,
\end{eqnarray*}
where $\lambda(K,Q) := |K'||Q| + \varepsilon (K) + \varepsilon (Q)$. Now (\ref{E:ber}) is equivalent to the equation
\begin{equation}\label{E:equiv}
    \sum_I  \int (-1)^{|I|}f_I^I \, \mu_\star = \sum_{K,Q}\int(-1)^{\lambda(K,Q)}  (f_K^P \star u_{PQ}) \tau_{K' Q'}\mu_\star.
\end{equation}
We will use the following identity proved in \cite{LMP2} (see also Proposition~ \ref{P:bert}),
\[
      \int_M f  g \, \mu_\star = \int_M f \star \I_\star g \, \mu_\star,
\]
where $\I_\star$ is the formal Berezin transform of the star product $\star$ and $f$ or $g$ has compact support. We have
\begin{eqnarray}\label{E:long}
       \nonumber\sum_{K,Q}\int (-1)^{\lambda(K,Q)} (f_K^P \star u_{PQ}) \tau_{K' Q'}\mu_\star =\\
  \sum_{K,Q}\int (-1)^{\lambda(K,Q)} f_K^P \star u_{PQ} \star \I_\star \tau_{K' Q'}\mu_\star =\\
   \nonumber\sum_{K,Q}\int (-1)^{\lambda(K,Q)} f_K^P\,  \I_\star^{-1} \left(u_{PQ} \star \I_\star \tau_{K' Q'}\right)\mu_\star
\end{eqnarray}
Taking into account that $\I_\star 1 = 1$, we see from (\ref{E:equiv}) and (\ref{E:long}) that (\ref{E:ber}) is equivalent to the equation
\[
       \sum_{K,Q}(-1)^{|K| + \lambda(K,Q)}u_{PQ} \star \I_\star \tau_{K' Q'} = \delta_P^K.
\]
Therefore,
\[
         (-1)^{|K| + \lambda(K,Q)} \I_\star \tau_{K' Q'} = v^{QK},
\]
whence we get that
\begin{equation}\label{E:taudens}
	        \tau_{KQ} = (-1)^{|K'| + \lambda(K',Q')} \I_\star^{-1} v^{Q'K'}.
\end{equation}
The statement of the theorem follows.
\end{proof}

\section{A star product on $U \times \C^{0|d}$}\label{S:superstar}
Given a possibly degenerate star product with separation of variables $\star$ on an open set $U \subset \C^m$, we introduce a class of admissible functions on $U \times \C^{0|d}$  for which the associated product $\ast$ is a star product. First we prove two technical statements.
\begin{proposition}\label{P:expq}
Let $(a_{\alpha \gamma}),(b_{\alpha\beta}), \mbox{ and } (c_{\beta\delta})$ be $d \times d$-matrices with constant coefficients and $(b_{\alpha\beta})$ be nondegenerate.
Consider the function
\[
      w = w_{IJ} \theta^I \bar\theta^J = e^Z
\]
on $\C^{0|d}$, where
\begin{equation*}
   Z = \frac{1}{2}a_{\alpha\gamma}\theta^\alpha \theta^\gamma + b_{\alpha\beta} \theta^\alpha \bar \theta^\beta + \frac{1}{2} c_{\beta\delta} \bar\theta^\beta \bar \theta^\delta. 
\end{equation*}
Then the $2^d \times 2^d$-matrix $(w_{IJ})$ is nondegenerate.
\end{proposition}
\begin{proof}
Let $\Gamma$ be the space of holomorphic Grassmann operators on the functions on $\C^{0|d}$ of the form
\[
A: f_{IJ} \theta^I \bar \theta^J \mapsto A_K^I f_{IJ} \theta^K \bar\theta^J.
\]
The mapping
\[
       \Gamma \ni A \mapsto Aw = A_K^I w_{IJ} \theta^K \bar\theta^J
\]
is a linear isomorphism of $\Gamma$ onto $C^\infty(\C^{0|d})$ if and only if the matrix $(w_{IJ})$ is nondegenerate. Denote by $\{\zeta_\alpha\}$ odd holomorphic variables dual to $\{\theta^\alpha\}$.  For $I= \{\alpha_1,\ldots, \alpha_k\}$ we write $\zeta_I = \zeta_{\alpha_k} \ldots \zeta_{\alpha_1}$. Let $\Pi$ be the space of functions
\begin{equation}\label{E:funpi}
p = p(\theta,\zeta) = p_K^I \theta^K \zeta_I,
\end{equation}
where $p_K^I$ are constants. To each function (\ref{E:funpi}) we relate the differential operator
\[
     \hat p := p_K^I  {\theta^K} {\frac{\p}{\p \theta^I}} \in \Gamma.
\]
The mapping $p \mapsto \hat{p}$ is a linear isomorphism of $\Pi$ onto $\Gamma$. Given a function $p(\theta,\zeta) \in \Pi$, the operator
\[
    A= e^{\frac{1}{2}a_{\alpha\gamma}\theta^\alpha \theta^\gamma} \hat{p}\, e^{-\frac{1}{2}a_{\alpha\gamma}\theta^\alpha \theta^\gamma}
\]
maps the function $w$ to the function
\[
            Aw = e^{\frac{1}{2}a_{\alpha\gamma}\theta^\alpha \theta^\gamma} p(\theta, b \bar\theta) \exp\left(b_{\alpha\beta} \theta^\alpha \bar \theta^\beta + \frac{1}{2}c_{\beta\delta} \bar\theta^\beta \bar \theta^\delta\right) = p(\theta, b \bar\theta) w.                  
\]
In the notation $p(\theta, b \bar\theta)$ the substitution $\zeta_\alpha = b_{\alpha\beta}\bar \theta^\beta$ is  implied. The mapping $A \mapsto Aw$ is a linear isomorphism from $\Gamma$ onto $C^\infty(\C^{0|d})$, because the matrix $(b_{\alpha\beta})$ is nondegenerate. It follows that the matrix $(w_{IJ})$ is nondegenerate.
\end{proof}

Assume that the matrix $(b_{\alpha\beta})$ from Proposition \ref{P:expq} is nondegenerate and denote by $(t^{JI})$ the inverse matrix of $(w_{IJ})$.

\begin{lemma}\label{L:tdd}
  If the matrix $(b_{\alpha\beta})$ is nondegenerate, then
\[
                   t^{[d][d]} = (-1)^{\frac{d(d-1)}{2}}\det b^{-1}.
\]
\end{lemma}
\begin{proof}
   We have from (\ref{E:delk}),
\[
   \bar\theta^L =  t^{LK} w_{KJ} \bar \theta^J = t^{LK}\delta_K w =  \left\{ e^Z t^{LK} \left(e^{-Z} \frac{\p}{\p \theta^K} e^Z\right)1 \right\} \Bigg|_{\theta=0},
\]
whence
\[
     t^{LK}  \left\{ \left(e^{-Z}\frac{\p}{\p \theta^K} e^Z \right)1\right\}\Bigg|_{\theta=0} = \bar \theta^L e^{- \frac{1}{2}c_{\beta\delta}\bar\theta^\beta\bar\theta^\delta}.
\] 
For $L=[d]$,
\begin{equation}\label{E:ld}
      t^{[d]K}  \left\{ \left(e^{-Z}\frac{\p}{\p \theta^K} e^Z \right)1 \right\}\Bigg|_{\theta=0} = \bar \theta^{[d]}.
\end{equation}
For $K = \{\alpha_1, \ldots, \alpha_n\}$,
\begin{equation}\label{E:thpr}
    \left ( e^{-Z}\frac{\p}{\p \theta^K} e^Z \right) 1= \left(\frac{\p}{\p \theta^{\alpha_n}} +  \frac{\p Z}{\p \theta^{\alpha_n}}\right) \ldots \left(\frac{\p}{\p \theta^{\alpha_1}} +  \frac{\p Z}{\p \theta^{\alpha_1}}\right)1.
\end{equation}
Since $Z$ is quadratic in the variables $\theta, \bar\theta$, the component of (\ref{E:thpr}) of degree $d$ in these variables is 
\begin{equation}\label{E:pzpz}
    \frac{\p Z}{\p \theta^d} \ldots \frac{\p Z}{\p \theta^1}
\end{equation}
if $n =d$ and zero otherwise. Formulas (\ref{E:ld}) and (\ref{E:pzpz}) imply that
\[
      t^{[d][d]} \left( b_{d \beta_d} \bar \theta^{\beta_d} \ldots b_{1 \beta_1} \bar \theta^{\beta_1} \right) = \bar \theta^{[d]},
\]
whence the lemma follows.
\end{proof}

We call an even nilpotent formal function $Y = \nu^{-1}Y_{-1} + Y_0 + \nu Y_1 + \ldots$ on $U \times \C^{0|d}$ a nondegenerate nilpotent potential if the matrix
\begin{equation}\label{E:ytheta}
     \left(\overrightarrow{\frac{\p}{\p \theta^\alpha}} Y_{-1} \overleftarrow{\frac{\p}{\p \bar \theta^\beta}}   \right)                
\end{equation}
is nondegenerate at every point of $U$.  The component of $Y_{-1}$ of degree two in the variables $\theta,\bar\theta$ can be written as
\begin{equation}\label{E:deg2}
                    \frac{1}{2}a_{\alpha\gamma}\theta^\alpha \theta^\gamma + b_{\alpha\beta} \theta^\alpha \bar \theta^\beta + \frac{1}{2} c_{\beta\delta} \bar\theta^\beta \bar \theta^\delta,
\end{equation}
where $a_{\alpha \gamma}, b_{\alpha\beta}, c_{\beta\delta} \in C^\infty(U)$. The matrix (\ref{E:ytheta}) is nondegenerate if and only if the matrix $b = (b_{\alpha\beta})$ is nondegenerate.
\begin{theorem}
Given an open set $U \subset \C^m$ and a nondegenerate nilpotent potential $Y$ on $U \times \C^{0|d}$, the formal function $e^Y = u_{PQ}\theta^P \bar\theta^Q$ on $U \times \C^{0|d}$ is admissible for any (possibly degenerate) star product with separation of variables $\star$ on $U$. The  leading term of the entry $v^{[d][d]}$ of the matrix $(v^{QP})$ inverse to $(u_{PQ})$ over the algebra $(C^\infty(U)[\nu^{-1},\nu]],\star)$ is 
\begin{equation}\label{E:lead}
                      (-1)^{\frac{d(d-1)}{2}}\det \left(b^{-1}\right)  \, \nu^d,
\end{equation}
where $b = (b_{\alpha\beta})$ is as in (\ref{E:deg2}).
\end{theorem}
\begin{proof}
The function $Y$ can be written as 
\[
     Y(\theta,\bar\theta) = \sum_{r = -1}^\infty \nu^r\,  Y_{r, PQ}\,  \theta^P \bar \theta^Q,
\]
where $Y_{r, PQ} \in C^\infty(U)$. Introduce formal odd variables
\[
     \eta^\alpha := \frac{1}{\sqrt{\nu}} \, \theta^\alpha \mbox{ and } \bar \eta^\beta := \frac{1}{\sqrt{\nu}}\,  \bar \theta^\beta
\]
and define a function
\[
     \tilde Y (\eta, \bar \eta) := Y(\sqrt{\nu}\, \eta, \sqrt{\nu}\,  \bar \eta) = \sum_{r = -1}^\infty \sum_{P,Q} \nu^{r+\frac{1}{2}(|P|+|Q|)}\, Y_{r, PQ}\, \eta^P \bar\eta^Q.
\]
Since $Y$ is even and nilpotent, $\tilde Y$ is a formal series in nonnegative integer powers of $\nu$, $\tilde Y = \tilde Y_0 + \nu \tilde Y_1 + \ldots$, and it follows from (\ref{E:deg2}) that
\[
     \tilde Y_0 = \frac{1}{2}a_{\alpha\gamma}\eta^\alpha \eta^\gamma + b_{\alpha\beta} \eta^\alpha \bar \eta^\beta + \frac{1}{2} c_{\beta\delta} \bar\eta^\beta \bar \eta^\delta.
\]
Set
\[
          \tilde u_{PQ} := \nu^{\frac{1}{2}(|P|+|Q|)} u_{PQ}.    
\]
Then
\[
        e^{\tilde Y} = \tilde u_{PQ} \eta^P \bar \eta^Q
\]
and therefore
\[
\tilde u_{PQ} \in C^\infty(U)[[\nu]].
\]
Denote by $w_{PQ}$ the coefficient at the zeroth power of $\nu$ of $\tilde u_{PQ}$, so that
\[
      e^{\tilde Y_0} = w_{PQ} \eta^P \bar\eta^Q.      
\] 
Let $\star$ be any  star product with separation of variables on $U$. The matrix $(u_{PQ})$ is invertible over the algebra $(C^\infty(U)[\nu^{-1},\nu]], \star)$ if and only if the matrix $(\tilde u_{PQ})$ is invertible over that algebra. 
The matrix  $(\tilde u_{PQ})$ is invertible over the algebra $(C^\infty(U)[[\nu]], \star)$ if and only if the matrix $(w_{PQ})$ is nondegenerate at every point of $U$, which is the case according to Proposition~ \ref{P:expq}. Therefore, the matrix $(u_{PQ})$ has an inverse, $(v^{QP})$, over the algebra $(C^\infty(U)[\nu^{-1},\nu]], \star)$. Denote by $(\tilde v^{QP})$ the inverse matrix of $(\tilde u_{PQ})$ over $(C^\infty(U)[[\nu]], \star)$. Clearly,
\[
                     v^{QP} = \nu^{\frac{1}{2}(|P|+|Q|)} \tilde v^{QP}.
\]
It follows from Lemma \ref{L:tdd} that the leading term of $v^{[d][d]}$ is (\ref{E:lead}), which concludes the proof of the theorem.
\end{proof}
\begin{corollary}\label{C:cantrden}
Let $\star$ be a nondegenerate star product with separation of variables on an open subset $U \subset \C^m$, $Y$ be a nondegenerate nilpotent potential  on $U \times \C^{0|d}$, and $\rho = \rho_{IJ} \theta^I \bar \theta^J$ be the canonical supertrace density function of the product $\ast$ associated with the pair $(\star, e^Y)$ as in ~(\ref{E:murho}). Then the  leading term of the component $\rho_{\emptyset\emptyset}$ of $\rho$ of degree zero with respect to the variables $\theta,\bar\theta$ is
\[
       \nu^{d-m} \psi,                    
\]
where $\psi \in C^\infty(U)$ is nonvavishing on $U$.
\end{corollary}
\begin{proof}
The corollary follows from formulas (\ref{E:trdens}) and (\ref{E:taudens}).
\end{proof}

In \cite{CMP3} it was proved that a possibly degenerate star product with separation of variables on a complex manifold is natural. Below we give a more elementary proof of this fact.
\begin{proposition}\label{P:nat}
Any star product with separation of variables on a complex manifold is natural. 
\end{proposition}
\begin{proof}
Let $\star$ be a star product with separation of variables on a complex manifold $M$ and $U \subset M$ be a holomorphic coordinate chart. Set $\L(U) := \{L_f | f \in C^\infty(U)[[\nu]]\}$. Given $f ,g \in C^\infty(U)[[\nu]]$, we have that
\[
      \nu^{-1}(f \star g - g \star f) \in C^\infty(U)[[\nu]].
\]
Therefore, for $A, B \in \L(U)$ we have that $ \nu^{-1}[A,B] \in \L(U)$. We will prove by induction on $n$ that for every $A = A_0 + \nu A_1 + \ldots \in \L(U)$ the order of $A_r$ for $r \leq n$ is not greater than $r$. This statement is true for $n=0$, because $A_0$ is a pointwise multiplication operator. The operators from $\L(U)$ do not contain antiholomorphic derivatives. Let $a$ be any holomorphic function on $U$. Then $L_a = a \in \L(U)$. In particular, for every  $A = A_0 + \nu A_1 + \ldots \in \L(U)$ we have that
\[
        \nu^{-1}[A,a] = \sum_{r = 1}^\infty \nu^{r-1} [A_r,a] \in \L(U).
\]
If we assume that the induction assumption holds for $n-1$, then the order of the operator $[A_n,a]$ is not greater than $n-1$ for every holomorphic function $a$. Therefore, the order of $A_n$ is not greater than $n$, which concludes the proof that for any $f \in C^\infty(M)[[\nu]]$ the operator $L_f$ is natural. The proof that $R_f$ is natural is similar.
\end{proof}

Let $U$  be an open subset of $\C^m$, $g^{lk}$ be a Poisson tensor of type $(1,1)$ on $U$, and $\{\xi_k, \bar \xi_l\}$ be the fiber coordinates on the cotangent bundle $T^\ast U$ dual to $\{z^k, \bar z^l\}$. Given a differential operator $D$ of order not greater than $r$ on $U$, let  $Symb_r(D) \in C^\infty(U)[\xi, \bar\xi]$ denote the principal symbol of order $r$ of the operator $D$  (which is  homogeneous in the variables $\xi, \bar \xi$ of degree $r$). If $A = A_0 + \nu A_1 + \ldots$ is a natural formal differential operator on $U$, set
\[
                 Symb(A) := \sum_{r=0}^\infty Symb_r(A_r) \in C^\infty(U)[[\xi, \bar\xi]].
\]
Let $\star$ be a star product with separation of variables on the Poisson manifold $(U, g^{lk})$. Given a function $h \in C^\infty(U)$, the operator $ L^\star_h$ of left $\star$-multiplication by $h$ is natural.
Set
\[
       Sh := Symb( L^\star_h).
\]
The mapping $C^\infty(U) \ni h \mapsto Sh$ is the source mapping of the formal symplectic groupoid of the star product $\star$ (see \cite{CMP3}).
It was proved in \cite{CMP3} that 
\[
                              (Sh)(z, \bar z, \xi) = \left\{\exp \left(\xi_k g^{lk} \frac{\p}{\p \bar z^l}\right)\right\} h.
\]
Let $Y$ be a nondegenerate nilpotent potential on $U \times \C^{0|d}$ and $\ast$ be the star product on $U \times \C^{0|d}$ associated with the pair $(\star, e^Y)$. Given a function $f \in C^\infty(U \times \C^{0|d})[\nu^{-1}, \nu]]$, the operator (\ref{E:oplf})  can be written uniquely as
\begin{equation}\label{E:lh}
              L_f = e^{-Y} L^\star_{h_I^K} \left(\nu^{|K|}{\theta^I} {\frac{\p}{\p \theta^K}}\right)e^Y
\end{equation}
for some $h_I^K \in C^\infty(U)[\nu^{-1},\nu]]$. If $h_I^K \in C^\infty(U)[[\nu]]$, then $L_f$ is a natural operator.

\begin{proposition}\label{P:lnat}
Given $f \in C^\infty(U \times \C^{0|d})[\nu^{-1},\nu]]$,  the operator $L_f$ is natural if and only if $f$ does not contain terms with negative powers of ~$\nu$.
\end{proposition}
\begin{proof}
If $L_f$ is natural, then $f = L_f 1$ does not contain terms with negative powers of $\nu$.  If $f \in C^\infty(U \times \C^{0|d})[[\nu]]$, assume that 
\[
                h_I^K =  \nu^n h^K_{I,n} + \nu^{n+1} h^K_{I,n+1}  + \ldots
\]
for some $n\in \Z$ and $(h^K_{I,n})$ is a nonzero matrix-valued function on~ $U$.  Then
\begin{equation}\label{E:numin}
\nu^{-n} L_f = A_0 + \nu A_1 + \ldots
\end{equation}
is a natural operator. Applying the operator $\nu^{-n}L_f$ to the unit constant and setting $\nu=0$, we get from (\ref{E:lh}) that the component $A_0$ in (\ref{E:numin})  is the left multiplication operator by the function
\[
                F:=  \left\{ \left(S\, h^K_{I,n}\right)\! \! \left (z,\bar z, \frac{\p Y_{-1}}{\p z}\right)\right \} \theta^I \left({\frac{\p Y_{-1}}{\p \theta}}\right)_K,
\]
where for $K = \{\alpha_1, \ldots, \alpha_p\}$ we use the notation
\[
     \left({\frac{\p Y_{-1}}{\p \theta}}\right)_K := {\frac{\p Y_{-1}}{\p \theta^{\alpha_p}}} \ldots {\frac{\p Y_{-1}}{\p \theta^{\alpha_1}}}.
\]
The function $F$ is well defined because $Y_{-1}$ is nilpotent and even. We will show that $F$ is a nonzero function. Denote by $F_{p,q}$ the component   of $F$ of degree $p$ in the variables $\theta$ and of degree $q$ in the variables $\bar\theta$ and set
\[
     F_r = \sum_{p+q=r} F_{p,q}.
\]
Let $N$ be the largest integer such that $h^K_{I,n} =0$ for all $I,K$ satisfying $|I|+|K| < N$. Then $F_r=0$ for $r < N$ and
\[
    F_N =  \sum_{|I|+|K|=N} h^K_{I,n} \theta^I (a \theta + b \bar\theta)_K,
\]
where for $K = \{\alpha_1, \ldots, \alpha_p\}$ we set
\[
           (a \theta + b \bar\theta)_K := \left( a_{\alpha_p \gamma_p} \theta^{\gamma_p} + b_{\alpha_p \beta_p} \bar \theta^{\beta_p}\right) \ldots  \left( a_{\alpha_1 \gamma_1} \theta^{\gamma_1} + b_{\alpha_1 \beta_1} \bar \theta^{\beta_1}\right).
\] 
Let $q$ be the least integer such that $h^K_{I,n} =0$ for all $I,K$ such that $|I|+|K|= N$ and $|K| >q$. Then
\[
     F_{N-q,q} = \sum_{|I|=N-q, |K|=q} h^K_{I,n} \theta^I (b \bar\theta)_K.
\]
Since $h^K_{I,n}$ is a nonzero function on $U$ for at least one pair $(I,K)$ with $|I|=N-q$ and $|K|=q$ and the matrix $(b_{kl})$ is  nondegenerate, we see that $F_{N-q,q}$ is nonzero, which implies that $F$ is nonzero. Applying (\ref{E:numin}) to the unit constant 1 and multiplying both sides by $\nu^n$, we get that
\[
     f = \nu^n F + \nu^{n+1} A_1 1 + \nu^{n+2} A_2 1+ \ldots 
\]
Since $F$ is nonzero, the assumption that $f \in C^\infty(U \times \C^{0|d})[[\nu]]$ implies that $n \geq 0$. It follows from  (\ref{E:numin}) that $L_f$ is a natural operator.
\end{proof}

{\it Remark.} The fact that $F$ is nonzero can be derived from the implicit function theorem in superanalysis (see \cite{Ber}).

\begin{corollary}\label{C:nat}
Given a possibly degenerate star product with separation of variables $\star$ on an open subset $U \subset \C^m$ and a nondegenerate nilpotent potential $Y$ on $U \times \C^{0|d}$, the product $\ast$ associated with the pair $(\star, e^Y)$  is a star product. Moreover, $\ast$ is a natural star product.
\end{corollary}
\begin{proof}
    Given an arbitrary function $f \in C^\infty(U \times \C^{0|d})$, the operator $L_f = A_0 + \nu A_1 + \ldots$ is natural according to Proposition \ref{P:lnat}. Therefore,  $A_0$ is a left multiplication operator and in (\ref{E:star}) $s = 0$. Since $L_f 1 = f \ast 1 = f$, we see that $A_0$ is the left multiplication operator by $f$ and thus $C_0(f,g) = fg$. Hence, $\ast$ is a star product. Given $f \in C^\infty(U \times \C^{0|d})$, one can prove similarly that $R_f$  is a natural operator. It follows that the star product $\ast$ is natural.
\end{proof}
 
{\it Remark.} One can generalize the proof of Proposition \ref{P:nat} in order to show that any star product with separation of variables on $U \times \C^{0|d}$ is natural. 

For a star product with separation of variables $\ast$ of the anti-Wick type on $U \times \C^{0|d}$ the operator $C_1$ is of the form
\begin{eqnarray}\label{E:cone}
     C_1(f,g) = \frac{\p f}{\p \bar z^l} A^{lk} \frac{\p g}{\p z^k} + \frac{\p f}{\p \bar z^l} B^{l \alpha} \overrightarrow{\frac{\p}{\p \theta^\alpha}}g +\\
 f \overleftarrow{\frac{\p}{\p \bar \theta^\beta}} C^{\beta k} \frac{\p g}{\p z^k} + f \overleftarrow{\frac{\p}{\p \bar \theta^\beta}} D^{\beta\alpha}\overrightarrow{\frac{\p}{\p \theta^\alpha}}g,\nonumber 
\end{eqnarray}
where the matrix
\begin{equation}\label{E:superg}
\left(\begin{array}{cc}
A^{lk} & B^{l\alpha} \\
C^{\beta k} & D^{\beta\alpha}
\end{array}\right)
\end{equation}
is an even Poisson tensor of type (1,1) on $U \times \C^{0|d}$ (so that $A^{lk}, D^{\beta\alpha}$ are even and $B^{l \alpha},C^{\beta k}$ are odd). A star product with separation of variables $\ast$ of the anti-Wick type on $U \times \C^{0|d}$  is called nondegenerate if the matrix (\ref{E:superg}) is nondegenerate at every point of $U$ (i.e., when the matrices $A^{lk}$ and $D^{\beta\alpha}$ are nondegenerate). Formula (\ref{E:cone}) and the definition of nondegeneracy have obvious analogues for the 
star products of the Wick type.

Let $U \subset \C^m$ be an open subset. Given an even formal function
\[
     X = \nu^{-1}X_{-1} + X_0 + \nu X_1 + \ldots  
\]
on  $U \times \C^{0|d}$ whose nilpotent component is a nondegenerate nilpotent potential $Y$ and such that the function
\[
    \Phi := X - Y 
\]
is a potential of a formal pseudo-K\"ahler form 
\[
    \omega = \nu^{-1}\omega_{-1} + \omega_0 + \ldots
\]
on $U$, we will call $X$ a nondegenerate formal potential on $U \times \C^{0|d}$. To each nondegenerate formal potential $X$ on $U \times \C^{0|d}$ there corresponds the star product $\ast$ on $U \times \C^{0|d}$ associated with the pair $(\star, e^Y)$, where 
$Y$ is the nilpotent component of $X$ and $\star$ is the star product with separation of variables on $U$ with the classifying form $\omega = i\p\bar\p(X-Y)$. 

{\it Remark.} One can similarly construct a star product with separation of variables of the Wick type starting from a nondegenerate formal potential $X$ (the notion of a nondegenerate potential will remain the same).

We will give explicit formulas for some left and graded right $\ast$-multi\-plication operators analogous to (\ref{E:classleft}) and (\ref{E:classright}). It follows from  (\ref{E:classleft}) and (\ref{E:oplf})  that
\begin{equation*}
    e^{-Y} L^\star_{\frac{\p \Phi}{\p z^k}} e^Y = e^{-Y}\left(\frac{\p \Phi}{\p z^k} + \frac{\p}{\p z^k}\right)e^Y = \frac{\p X}{\p z^k} + \frac{\p}{\p z^k}
\end{equation*}
and
\begin{equation*}
     e^{-Y}\frac{\p}{\p \theta^\alpha} e^Y = e^{-X} \frac{\p}{\p \theta^\alpha} e^X = \frac{\p X}{\p \theta^\alpha} + \frac{\p}{\p \theta^\alpha}.
\end{equation*}
Therefore,
\begin{equation}\label{E:l}
        L_{\frac{\p X}{\p z^k}} = \frac{\p X}{\p z^k} + \frac{\p}{\p z^k} \mbox{ and } L_{\frac{\p X}{\p \theta^\alpha}} =  \frac{\p X}{\p \theta^\alpha}  + \frac{\p}{\p \theta^\alpha}.
\end{equation}
Similarly, it follows from (\ref{E:classright}) and  (\ref{E:oprt})  that
\begin{equation}\label{E:r}
     R_{\frac{\p X}{\p \bar z^l}} = \frac{\p X}{\p \bar z^l} + \frac{\p}{\p \bar z^l} \mbox{ and } R_{\frac{\p X}{\p \bar\theta^\beta}} =\frac{\p X}{\p \bar\theta^\beta} + \frac{\p}{\p \bar\theta^\beta}.
\end{equation}
\begin{proposition}\label{P:nondeg}
Suppose that $\ast$ is a star product with separation of variables on $U \times \C^{0|d}$ such that there exists a formal function 
\[
X = \nu^{-1}X_{-1} + X_0 + \nu X_1 + \ldots
\]
on $U \times \C^{0|d}$ for which formulas (\ref{E:l}) (or formulas (\ref{E:r})) hold. Then $X$ is a nondegenerate formal potential and $\ast$ is nondegenerate.
\end{proposition}
\begin{proof}
Formulas (\ref{E:cone}) and  (\ref{E:l}) (or (\ref{E:r})) imply that the matrix (\ref{E:superg}) is inverse to the matrix
\begin{equation}\label{E:hess}
         \left(\begin{array}{cc}
\frac{\p^2 X_{-1}}{\p z^k \p \bar z^l} & \frac{\p X_{-1}}{\p z^k}\overleftarrow{\frac{\p}{\p \bar \theta^\beta}}\\
 \overrightarrow{\frac{\p}{\p \theta^\alpha}}\frac{\p X_{-1}}{\p  \bar z^l} & \overrightarrow{\frac{\p}{\p \theta^\alpha}} X_{-1}\overleftarrow{\frac{\p}{\p \bar \theta^\beta}}
\end{array}\right)                 
\end{equation}
and therefore both matrices are nondegenerate, whence the proposition follows.
\end{proof}

\begin{proposition}\label{P:bullast}
Let $\bullet$ be another star product with separation of variables  on $U \times \C^{0|d}$ for which the operators of graded right $\bullet$-multiplication by $\p X / \p \bar z^l$ and $\p X / \p \bar\theta^\beta$ are the same as for the product $\ast$,
\[
     R^\bullet_{\frac{\p X}{\p \bar z^l}} = R_{\frac{\p X}{\p \bar z^l}} \mbox { and } Q^\bullet_{\frac{\p X}{\p \bar\theta^\beta}} = R_{\frac{\p X}{\p \bar\theta^\beta}}.
\]
Then the star products $\bullet$ and $\ast$ coincide.
\end{proposition}
\begin{proof}
 Given $f \in C^\infty(U \times \C^{0|d})[\nu^{-1},\nu]]$, let $L^\bullet_f$ be the left $\bullet$-multipli\-cation operator by $f$. It supercommutes with the operators (\ref{E:r}) and with the operators of (left) multiplication by antiholomorphic functions on $U \times \C^{0|d}$. It follows from (\ref{E:r}) that the operator $e^Y L^\bullet_f e^{-Y}$ supercommutes with the operators
\[
                  \frac{\p \Phi}{\p \bar z^l} + \frac{\p}{\p \bar z^l}, \frac{\p}{\p \bar\theta^\beta},
\]
and the  multiplication operators by antiholomorphic functions. Hence, it can be written as
\[
          e^Y L^\bullet_f e^{-Y} = L^\star_{h_I^K} \theta^I \frac{\p}{\p \theta^K}
\]
for some $h_I^K \in C^\infty(U)[\nu^{-1},\nu]]$. It implies that $L^\bullet_f$ is a left $\ast$-multiplication operator. Since $L^\bullet_f 1 = f \bullet 1 = f$, it follows that $L^\bullet_f = L_f$. Therefore, the star products $\bullet$ and $\ast$ coincide.
\end{proof}

Proposition \ref{P:bullast} means that if $X$ is a nondegenerate formal potential on $U \times \C^{0|d}$, there exists a unique star product with separation of variables~ $\ast$ on $U \times \C^{0|d}$ satisfying (\ref{E:l}) and (\ref{E:r}).

Let $\tilde X$ be another nondegenerate formal potential on $U \times \C^{0|d}$ such that $\tilde X = X + a + b$, where the functions $a = a(z, \theta)$ and $b = b(\bar z, \bar\theta)$ are holomorphic and antiholomorphic, respectively. Lemma~ \ref{L:eab} implies that the star product with separation of variables on $U \times \C^{0|d}$ corresponding to the potential $\tilde X$ coincides with $\ast$. 

Let $E$ be a holomorphic vector bundle  over a complex manifold ~$M$. One can define differential forms on the split complex supermanifold ~$\Pi E$ and extend the operators $\p,\bar\p,$ and $d = \p + \bar\p$ to $\Pi E$ (see \cite{Ber},\cite{W}). If $U \subset M$ is a coordinate chart for which there is a holomorphic trivialization $\Pi E|_U \cong U \times \C^{0|d}$, a differential form on $U \times \C^{0|d}$ is a function in the supercommuting variables $z, \bar z, \theta, \bar \theta, dz, d\bar z, d\theta, d\bar \theta$.  The variables $z, \bar z, d\theta, d\bar\theta$ are even and $\theta, \bar\theta, dz,d\bar z$ are odd. The variables $d\theta, d\bar\theta$ are not nilpotent and a differential form is supposed to be polynomial in $d\theta,d\bar\theta$. We assume that in local coordinates the operators
\begin{equation}\label{E:superdd}
     \p = dz^k \frac{\p}{\p z^k} + d \theta^\alpha \frac{\p}{\p \theta^\alpha} \mbox{ and } \bar\p = d\bar z^l \frac{\p}{\p \bar z^l} +  d\bar\theta^\beta\frac{\p}{\p \bar \theta^\beta}
\end{equation}
on $U \times \C^{0|d}$ act from the left.  If $U$ is contractible, using the Euler operators $\theta^\alpha \p / \p \theta^\alpha$ and  $\bar \theta^\beta \p / \p \bar \theta^\beta$, one can extend the $\bar\p$-Poincar\' e lemma to $U \times \C^{0|d}$. In particular, if
\[
      \Omega = A_{kl} dz^k d\bar z^l + B_{k\beta} dz^k d \bar \theta^\beta + C_{\alpha  l} d \theta^\alpha d \bar z^l + B_{\alpha  \beta} d \theta^\alpha d \bar \theta^\beta
\]
is a closed even (1,1)-form on $U \times \C^{0|d}$ with $U$ contractible, there exists an even potential $F$ on $U \times \C^{0|d}$ such that 
\begin{eqnarray*} 
\Omega = i\p \bar \p F = i\Bigg(\frac{\p^2 F}{\p z^k \bar z^l}  dz^k d\bar z^l - \frac{\p^2 F}{\p z^k \p \bar \theta^\beta} dz^k d \bar \theta^\beta +\\
 \frac{\p^2 F}{\p \theta^\alpha \p \bar z^l}d \theta^\alpha d \bar z^l + \frac{\p^2 F}{\p \theta^\alpha \p \bar \theta^\beta} d \theta^\alpha d \bar \theta^\beta\Bigg).
\end{eqnarray*} 
The potential $F$ is determined uniquely up to a summand $a + b$, where $a$ is holomorphic and $b$ is antiholomorphic. 

The form $\Omega$ is nondegenerate if the matrix
\begin{equation}\label{E:abcd}
    \left(\begin{array}{cc}
A_{kl} & B_{k \beta}\\
C_{\alpha l} & D_{\alpha\beta}
\end{array}\right)  
= i  \left(\begin{array}{cc}
\frac{\p^2 F}{\p z^k \bar z^l} & - \frac{\p^2 F}{\p z^k \p \bar \theta^\beta}\\
\frac{\p^2 F}{\p \theta^\alpha \p \bar z^l} & \frac{\p^2 F}{\p \theta^\alpha \p \bar \theta^\beta}
\end{array}\right)
\end{equation}
is nondegenerate. Since the entries $A_{kl}$ and $D_{\alpha\beta}$ are even and the entries $B_{k \beta}$ and $C_{\alpha l}$ are odd, the matrix (\ref{E:abcd}) is nondegenerate if and only if the matrices
\[
 (A_{kl}) =i \left(\frac{\p^2 F}{\p z^k \bar z^l}\right) \mbox{ and } (D_{\alpha\beta}) = i\left(\frac{\p^2 F}{\p \theta^\alpha \p \bar \theta^\beta} \right) 
\]
are nondegenerate.

If $\Omega$ is a closed even global (1,1)-form on $\Pi E$, then for every contractible neighborhood $U \subset M$ over which $E$ is holomorphically trivializable there exists an even potential $F$ on $\Pi E|_U$ satisfying $\Omega = i\p \bar \p F$.

Combining the results obtained above, we get the following theorem.
\begin{theorem}
Let $E$ be a holomorphic vector bundle of rank $d$ over a complex manifold $M$. Suppose that
\begin{equation}\label{E:Omega}
         \Omega = \nu^{-1} \Omega_{-1} + \Omega_0 + \nu \Omega_1 + \ldots
\end{equation}
is a closed formal (1,1)-form globally defined on $\Pi E$ such that $\Omega_{-1}$ is nondegenerate. Then there exists a unique nondegenerate star product with separation of variables $\ast$ on $\Pi E$ such that for any local holomorphic trivialization $\Pi E|_U \cong U \times \C^{0|d}$ over a contractible coordinate chart $U \subset M$ there exists a formal potential $X$ of the form $\Omega$ on $U \times \C^{0|d}$ for which formulas (\ref{E:l}) and (\ref{E:r}) hold.
\end{theorem}

{\it Remark.} One can prove that in complete analogy with the results of ~\cite{CMP1}, the nondegenerate formal forms (\ref{E:Omega}) on $\Pi E$ bijectively correspond to the nondegenerate star products with separation of variables on $\Pi E$.

\begin{corollary}
  Let (\ref{E:Omega}) be a closed nondegenerate formal (1,1)-form on $\Pi E$ and $\ast$ be the corresponding nondegenerate star product with separation of variables on $\Pi E$. Then there exists a globally defined canonical supertrace functional $\sigma$, given by a global Berezin supertrace density $\mu$,
\[
         \sigma(f) = \int f \, \mu,
\]
where $f$ is a formal function with compact support on $\Pi E$.
\end{corollary}
\begin{proof}
The corollary follows from Lemma \ref{L:sigind} and Theorem \ref{T:trdens}.
\end{proof}

In the next proposition we consider a star product $\ast$ with separation of variables {\it of the Wick type} on $U \times \C^{0|d}$. 
For a holomorphic function $a$ and an antiholomorphic function $b$ on $U \times \C^{0|d}$ the operator of left $\ast$-multiplication by $b$ and the operator of graded right $\ast$-multiplication by $a$ are left multiplication operators,
\[
                         L_b = b \mbox{ and } R_a = a.
\]
 
\begin{proposition}\label{P:wick}
Let $U$ be a contractible open subset of $\C^m$ and $\ast$ be a star product with separation of variables of the Wick type on $U \times \C^{0|d}$. Assume that
\[
     v_k = \nu^{-1} v_{k, -1} + v_{k,0} + \nu v_{k,1} + \ldots, 1 \leq k \leq m,
\]
are even and 
\[
    w_\alpha = \nu^{-1} w_{\alpha,-1} + w_{\alpha,0} + \nu w_{\alpha,1} + \ldots, 1 \leq \alpha \leq d,
\]
are odd formal functions on $U \times \C^{0|d}$, respectively,  such that they pairwise $\ast$-supercommute and satisfy the following $\ast$-supercommu\-tation relations for any holomorphic $a$,
\begin{equation}\label{E:vw}
       \left[v_k, a\right]_\ast = \frac{\p a}{\p z^k}  \mbox{ and } \left[w_\alpha, a\right]_\ast = \frac{\p a}{\p \theta^\alpha}.
\end{equation}
Then the star product $\ast$ is nondegenerate and there exists a nondegenerate potential $X = \nu^{-1}X_{-1} + X_0 + \ldots$ on $U \times \C^{0|d}$ such that
\begin{equation}\label{E:potX}
     v_k = -\frac{\p X}{\p z^k} \mbox{ and } w_\alpha = -\frac{\p X}{\p \theta^\alpha}.
\end{equation}
For this potential,
\begin{equation}\label{E:rwick}
     R_{\frac{\p X}{\p z^k}} = \frac{\p X}{\p z^k} + \frac{\p}{\p z^k} \mbox{ and } R_{\frac{\p X}{\p \theta^\alpha}} = \frac{\p X}{\p \theta^\alpha} + \frac{\p}{\p \theta^\alpha}.
\end{equation}
\end{proposition}
\begin{proof}
We get from relations (\ref{E:vw})  that
\[
     L_{v_k} a - R_{v_k} a = \frac{\p a}{\p z^k}  \mbox{ and } L_{w_\alpha} a - R_{w_\alpha} a = \frac{\p a}{\p \theta^\alpha}.
\]
Since $R_f$ contains only holomorphic partial derivatives $\p /\p z^k$ and $\p / \p \theta^\alpha$, we have
\begin{equation}\label{E:rvw}
                    R_{v_k} = v_k - \frac{\p}{\p z^k} \mbox{ and } R_{w_\alpha} = w_\alpha - \frac{\p}{\p \theta^\alpha}.
\end{equation}

Since the functions $v_k, 1 \leq k \leq m,$ and $w_\alpha, 1 \leq \alpha \leq d,$ pairwise $\ast$-supercommute, they satisfy the equations
\[
     \frac{\p v_k}{\p z^p} = \frac{\p v_p}{\p z^k}, \frac{\p v_k}{\p \theta^\alpha} = \frac{\p w_\alpha}{\p z^k}, \mbox{ and }\frac{\p w_\alpha}{\p \theta^\gamma} = - \frac{\p w_\gamma}{\p \theta^\alpha},
\]
which are equivalent to the condition that the form
\[
                   \Lambda = v_k dz^k + w_\alpha d\theta^\alpha
\]
 on $U \times \C^{0|d}$ is $\p$-closed. By the $\bar \p$-Poincar\' e lemma on $U \times \C^{0|d}$, there is a formal function $X = \nu^{-1}X_{-1} + X_0 + \ldots$ such that $\Lambda = - \p X$. Thus, $X$ satisfies (\ref{E:potX}). It follows from (\ref{E:rvw}) that equations (\ref{E:rwick}) hold. By Proposition \ref{P:nondeg}, both the potential $X$ and the star product $\ast$ are nondegenerate.
\end{proof}

\section{Formal Berezin transform}

Given a possibly degenerate star product with separation of variables of the anti-Wick type $\ast$ on a split supermanifold $\Pi E$, there exists a unique formal differential operator $\I = 1 + \nu\I_1 + \ldots$ on $\Pi E$ such that
\[
              \I (ba) = b \ast a
\]
for any local holomorphic function $a$ and antiholomorphic function $b$. Clearly, $\I a = a$ and $\I b = b$.  

We introduce a star product ~$\ast'$ on $\Pi E$ by the formula
\begin{equation}\label{E:astprime}
     f \ast' g := \I^{-1} \left(\I f \ast \I g\right). 
\end{equation}
\begin{lemma}
  The star product $\ast'$ is of the Wick type.
\end{lemma}
\begin{proof}
If functions $a, \tilde a$ are local holomorphic and  $b$ is local antiholomorphic on $\Pi E$, then for $f = ba$ we have
\begin{eqnarray*}
     f \ast' \tilde a = (ba) \ast' \tilde a = \I^{-1} (\I (ba) \ast \I \tilde a) =\\
 \I^{-1} (b \ast a \ast \tilde a) = \I^{-1} (b \ast (a\tilde a)) = ba \tilde a = f \tilde a.
\end{eqnarray*}
Therefore, $ f \ast' a =  f a$ for any function $f$ and any holomorphic $a$. Similarly, $b \ast' g = b g$ for any function $g$ and any antiholomorphic $b$. Therefore, $\ast'$ is a star product with separation of variables of the Wick type.
\end{proof}
\begin{proposition}\label{P:bert}
Given a holomorphic vector bundle $E$ over a complex manifold $M$, let (\ref{E:Omega}) be a closed nondegenerate formal (1,1)-form on $\Pi E$, $\ast$ be the corresponding nondegenerate star product with separation of variables on $\Pi E$, $\I$ be its formal Berezin transform, and $\mu$ be its canonical Berezin supertrace density. Then for formal functions $f,g$ on $\Pi E$ the following identities hold,
\begin{equation}\label{E:bertr}
       \int f \ast g \, \mu = \int f\,  \I^{-1}g \, \mu = \int (\I^{-1} f)\,  g \, \mu,
\end{equation}
where $f$ or $g$ has compact support.
\end{proposition}
\begin{proof}
  Suppose that there is a contractible coordinate chart $U \subset M$ such that $E$ is holomorphically trivializable over $U$, $f$ has compact support in $\Pi E|_U$ and $g = \I (ba) = b \ast a$, where $a$ is holomorphic and $b$ is antiholomorphic on $\Pi E|_U$. Assume also that $a, b$, and  $f$ are homogeneous. Then
\begin{eqnarray*}
     \int f \ast g \, \mu = \int f \ast b \ast a \, \mu = (-1)^{|a|(|f|+|b|)}\int a \ast f \ast b \, \mu =\\
  (-1)^{|a|(|f|+|b|)}\int afb \, \mu  = \int fba \, \mu =  \int f\,  \I^{-1}g \, \mu.
\end{eqnarray*}
Since $\I$ is a formal differential operator and $\ast$ is a differential product, we obtain the first equality in (\ref{E:bertr}) for $f$ with compact support in $\Pi E|_U$ and any $g$. Using a partition of unity on $M$ we can prove the first equality in (\ref{E:bertr}) for any $f$ with compact support and any $g$. If $g$ has compact support, we can assume that $f$ also has compact support, which proves the first equality. The second equality follows from the first one and the fact that $\mu$ is a supertrace density for $\ast$:
\[
    \int f \ast g \, \mu = (-1)^{|f||g|} \int g \ast f \, \mu = (-1)^{|f||g|} \int g \,  \I^{-1}f \, \mu = \int (\I^{-1} f)\,  g \, \mu.
\]
\end{proof}
Let $U \subset \C^m$ be a contractible open subset and $\ast$ be a nondegenerate star product with separation of variables  of the anti-Wick type on $U \times \C^{0|d}$ such that there exists a nondegenerate formal potential $X$ satisfying (\ref{E:l}).  It follows from (\ref{E:l}) that the elements $\p X / \p z^k, 1 \leq k \leq m,$ and $\p X / \p \theta^\alpha, 1 \leq \alpha \leq d,$ pairwise $\ast$-supercommute and that for a holomorphic function $a$ on $U \times \C^{0|d}$,
\begin{equation}\label{E:scr}
       \left[\frac{\p X}{\p z^k}, a\right]_\ast = \frac{\p a}{\p z^k} \mbox{ and } \left[\frac{\p X}{\p \theta^\alpha}, a\right]_\ast = \frac{\p a}{\p \theta^\alpha}.
\end{equation}
Denote by $\I$ the formal Berezin transform for the product $\ast$ and by ~$\ast'$ the corresponding star product of the Wick type given by (\ref{E:astprime}). Applying $\I$ to the $\ast$-supercommutation relations (\ref{E:scr}) we obtain that the elements
\[
                \I^{-1} \left( \frac{\p X}{\p z^k} \right), 1 \leq k \leq m, \mbox{ and }  \I^{-1} \left( \frac{\p X}{\p \theta^\alpha} \right), 1 \leq \alpha \leq d,
\]
pairwise $\ast'$-supercommute and satisfy the relations
\[
     \left[\I^{-1}\left(\frac{\p X}{\p z^k}\right), a\right]_{\ast'} = \frac{\p a}{\p z^k} \mbox{ and } \left[\I^{-1}\left(\frac{\p X}{\p \theta^\alpha}\right), a\right]_{\ast'} = \frac{\p a}{\p \theta^\alpha}.
\]
According to Proposition \ref{P:wick}, there exists a nondegenerate formal potential $X'$ on $U \times \C^{0|d}$ satisfying
\begin{equation}\label{E:xprime}
             \frac{\p X'}{\p z^k} = - \I^{-1}\left(\frac{\p X}{\p z^k}\right)  \mbox{ and } \frac{\p X'}{\p \theta^\alpha} = - \I^{-1}\left(\frac{\p X}{\p \theta^\alpha}\right)
\end{equation}
for which
\[
            R'_{\frac{\p X'}{\p z^k}} = \frac{\p X'}{\p z^k} + \frac{\p}{\p z^k} \mbox{ and } R'_{\frac{\p X'}{\p \theta^\alpha}} = \frac{\p X'}{\p \theta^\alpha} + \frac{\p}{\p \theta^\alpha},
\]
where $R'_f$ is the graded right $\ast'$-multiplication operator by $f$. It means that the star product of the Wick type $\ast'$ can be constructed from the potential $X'$ as in Section \ref{S:superstar} (with appropriate changes) and is nondegenerate. In particular,
\[
            L'_{\frac{\p X'}{\p \bar z^l}} = \frac{\p X'}{\p \bar z^l} + \frac{\p}{\p \bar z^l} \mbox{ and } L'_{\frac{\p X'}{\p \bar\theta^\beta}} = \frac{\p X'}{\p \bar\theta^\beta} + \frac{\p}{\p \bar \theta^\beta},
\]
where $L'_f$ is the left $\ast'$-multiplication operator by $f$.

\begin{theorem}
Let $U \subset \C^m$ be a contractible open subset and $\ast$ be a nondegenerate star product with separation of variables of the anti-Wick type on $U \times \C^{0|d}$ determined by a nondegenerate formal potential $X = \nu^{-1}X_{-1} + X_0 +  \ldots$. Then the following statements hold. (a) There exists a nondegenerate formal potential $X' = \nu^{-1}X'_{-1} + X'_0 +  \ldots$ which determines the star product with separation of variables of the Wick type (\ref{E:astprime}) and satisfies the equations
\begin{eqnarray}\label{E:allfour}
    \frac{\p X'}{\p z^k} = - \I^{-1}\left(\frac{\p X}{\p z^k}\right), \frac{\p X'}{\p \theta^\alpha} = - \I^{-1}\left(\frac{\p X}{\p \theta^\alpha}\right),\\
  \frac{\p X'}{\p \bar z^l} = - \I^{-1}\left(\frac{\p X}{\p \bar z^l}\right),
 \mbox{ and } \frac{\p X'}{\p \bar\theta^\beta} = - \I^{-1}\left(\frac{\p X}{\p \bar\theta^\beta}\right). \nonumber
\end{eqnarray}
(b) The Berezin density
\begin{equation}\label{E:std}
             e^{X + X'}dz d\bar z d\theta d\bar \theta
\end{equation}
is a supertrace density for the product $\ast$ on $U \times \C^{0|d}$.
\end{theorem}
\begin{proof}
Let $\mu = \rho\,  dz d\bar z d \theta d\bar \theta$ be the canonical Berezin supertrace density of $\ast$ on $U \times \C^{0|d}$, where $\rho$ is a formal function on $U \times \C^{0|d}$. It follows from Corollary ~\ref{C:cantrden} that there exists an even formal function $S$ on $U \times \C^{0|d}$ such that
\[
     \rho = \nu^{d-m} e^S.
\]
Set $\tilde X : = S - X$. Then $\tilde X = \nu^r \tilde X_r + \nu^{r+1} \tilde X_{r+1} + \ldots$ for some $r \in \Z$. We have from Proposition \ref{P:bert} and the first formula in (\ref{E:l})  that
\begin{eqnarray*}
\int \left(\I^{-1} \left(\frac{\p X}{\p z^p} \right) + \frac{\p \tilde X}{\p z^p}\right)f \, \mu = \int \left(\frac{\p X}{\p z^p} \ast f + \frac{\p \tilde X}{\p z^p}f\right)\mu =\\
\int \left(\frac{\p f}{\p z^p} + \frac{\p X}{\p z^p}f + \frac{\p \tilde X}{\p z^p}f \right) \mu = \nu^{d-m}\int \frac{\p}{\p z^p} \left(e^S f\right) dz d\bar z d \theta d\bar \theta = 0.
\end{eqnarray*}
Since $f$ is arbitrary, it follows that
\begin{equation}\label{E:first}
       \frac{\p \tilde X}{\p z^k} = - \I^{-1}\left(\frac{\p X}{\p z^k}\right).
\end{equation}
Similarly, we can derive from Proposition \ref{P:bert}, the second formula in (\ref{E:l}), and formulas (\ref{E:r}) that
\begin{eqnarray}\label{E:therest}
    \frac{\p \tilde X}{\p \theta^\alpha} = - \I^{-1}\left(\frac{\p X}{\p \theta^\alpha}\right),  \frac{\p \tilde X}{\p \bar z^l} = - \I^{-1}\left(\frac{\p X}{\p \bar z^l}\right),\\
 \mbox{ and } \frac{\p \tilde X}{\p \bar\theta^\beta} = - \I^{-1}\left(\frac{\p X}{\p \bar\theta^\beta}\right). \nonumber
\end{eqnarray}
Equations (\ref{E:first}) and (\ref{E:therest}) determine the function $\tilde X$ up to a formal constant summand. Moreover, $\tilde X_k$ is a constant for $k < -1$. According to~ (\ref{E:xprime}), 
\[
      X' := \nu^{-1} \tilde X_{-1} + \tilde X_0 + \ldots
\]
is a nondegenerate formal potential which determines the star product ~ (\ref{E:astprime}). Clearly, $X'$ satisfies the conditions of the theorem and (\ref{E:std}) is a Berezin supertrace density for the product $\ast$.
\end{proof}

We want to describe an important class of star products with separation of variables on split supermanifolds. Let $E$ be a holomorphic Hermitian vector bundle of rank $d$ over a pseudo-K\" ahler manifold $M$ and $\star$ be a star product with separation of variables on $M$. The Hermitian metric on $E$ determines a global function $h$ on $\Pi E$ by the following condition. Let $U \subset M$ be a coordinate chart  such that there exists a holomorphic trivialization $E|_U \cong U \times \C^d$ and $h_{\alpha \beta}$ be the Hermitian fiber metric over $U$. Then $h = h_{\alpha \beta}\theta^\alpha \bar \theta^\beta$  on $\Pi E|_U \cong U \times \C^{0|d}$. Notice that $Y = \nu^{-1}h_{\alpha \beta}\theta^\alpha \bar \theta^\beta$ is a nondegenerate nilpotent potential. Denote by $\ast$ the star product with separation of variables on $\Pi E$ associated with the pair $(\star, \exp\{\nu^{-1} h\})$. Writing in local coordinates
\[
     \exp\left\{\nu^{-1} h\right\} = u_{IJ} \theta^I \bar \theta^J,
\]
we see that the matrix $(u_{IJ})$ is block diagonal with the diagonal blocks formed by the entries $u_{IJ}$ with $|I|=|J| = k$ for each $k$ satisfying $0 \leq k \leq d$, because $u_{IJ} = 0$ for $|I| \neq |J|$. We have
\[
      u_{[d][d]} = \frac{1}{\nu^d} (-1)^{\frac{d(d-1)}{2}}  \det ( h_{\alpha\beta}).
\]
The block $v^{[d][d]}$ of the inverse matrix $(v^{JI})$ of $(u_{IJ})$ over the algebra $(C^\infty(U)[\nu^{-1},\nu]],\star)$ satisfies
\[
        u_{[d][d]} \star v^{[d][d]} =  v^{[d][d]} \star u_{[d][d]} =  1.
\]
This observation and formula (\ref{E:taudens}) allow to canonically normalize the supertrace density (\ref{E:std}) of the product $\ast$ in important applications.

\end{document}